\documentclass[11pt,a4paper,reqno]{amsart}  

\usepackage[utf8]{inputenc}
\usepackage[english]{babel}
\usepackage[margin=3cm]{geometry}

\usepackage[usenames,dvipsnames,svgnames,table]{xcolor}
\usepackage{float}
\usepackage{caption,subcaption}
\usepackage{tikz,pgf}

\usepackage{fancyhdr}
\usepackage{footmisc}
\usepackage{lipsum}

\usepackage{blindtext}

\usepackage{amsfonts,amsmath,amssymb,amsthm,mathtools,mathrsfs,bbm}
\usepackage{array,caption,subcaption,graphicx}
\usepackage[shortlabels]{enumitem}

\usepackage{hyperref}
\usepackage{xcolor}

\definecolor{unbleu}{rgb}{0.03, 0.15, 0.4}

\hypersetup{
pdfborder = {0 0 0},
colorlinks,
linkcolor=unbleu,
citecolor=unbleu,
urlcolor=unbleu
}

\usepackage{tikz-cd}

\theoremstyle{plain}
\newtheorem{theorem}{Theorem}[section]
\newtheorem*{theorem*}{Theorem}
\newtheorem{lemma}[theorem]{Lemma}
\newtheorem*{lemma*}{Lemma}
\newtheorem{proposition}[theorem]{Proposition}
\newtheorem*{proposition*}{Proposition}

\newtheorem*{corollary*}{Corollary}

\theoremstyle{definition}
\newtheorem{definition}[theorem]{Definition}

\newtheorem{remark}[theorem]{Remark}

\newcommand{\modl}[1]{\lvert #1 \rvert}
\newcommand\given[1][]{\:#1\middle|\:}

\DeclarePairedDelimiter{\floor}{\lfloor}{\rfloor}

\setlist[enumerate]{topsep=0pt}
\setlist[itemize]{topsep=0pt}
\numberwithin{equation}{section}

\usepackage[T2A,T1]{fontenc}
\DeclareSymbolFont{cyrillic}{T2A}{cmr}{m}{n}
\DeclareMathSymbol{\D}{\mathalpha}{cyrillic}{196}

\def\R{\ensuremath{\mathbb R}}
\def\N{\ensuremath{\mathbb N}}

\def\Z{\ensuremath{\mathbb Z}}
\def\I{\ensuremath{\mathcal{I}}}
\def\e{{\ensuremath{\rm e}}}

\def\B{\ensuremath{\mathcal B}}

\def\p{\ensuremath{\mathbb P}}

\def\X{\mathbb{X}}
\def\ie{{\em i.e.}, }

\def\E{\mathbb E}

\def\TT{\sigma}
\def\m{\mathfrak{m}}

\def\1{{\bf 1}}

\def\W{\ensuremath{\mathscr W}}

\def\Ind{\ensuremath{{\mathbbm{1}}}}

\def\sm{\setminus}
\def\eps{\varepsilon}

\def\proj{p}
\def\Proj{P}

\makeatletter
\@namedef{subjclassname@2020}{\textup{2020} Mathematics Subject Classification}
\makeatother

\begin{document}

\title{A functional limit theorem for a dynamical system with an observable maximised on a Cantor set}

\author[R. Couto]{Raquel Couto}
\address{Raquel Couto\\ Centro de Matem\'{a}tica \& Faculdade de Ci\^encias da Universidade do Porto\\ Rua do
Campo Alegre 687\\ 4169-007 Porto\\ Portugal}
\email{\href{mailto:up201106786@fc.up.pt}{up201106786@fc.up.pt}}

\author[A. C. M. Freitas]{Ana Cristina Moreira Freitas}
\address{Ana Cristina Moreira Freitas\\ Centro de Matem\'{a}tica \&
Faculdade de Economia da Universidade do Porto\\ Rua Dr. Roberto Frias \\
4200-464 Porto\\ Portugal} \email{\href{mailto:amoreira@fep.up.pt}{amoreira@fep.up.pt}}
\urladdr{\url{http://www.fep.up.pt/docentes/amoreira/}}

\author[J. M. Freitas]{Jorge Milhazes Freitas}
\address{Jorge Milhazes Freitas\\ Centro de Matem\'{a}tica \& Faculdade de Ci\^encias da Universidade do Porto\\ Rua do
Campo Alegre 687\\ 4169-007 Porto\\ Portugal}
\email{\href{mailto:jmfreita@fc.up.pt}{jmfreita@fc.up.pt}}
\urladdr{\url{http://www.fc.up.pt/pessoas/jmfreita/}}

\author[M. Todd]{Mike Todd}
\address{Mike Todd\\ Mathematical Institute\\
University of St Andrews\\
North Haugh\\
St Andrews\\
KY16 9SS\\
Scotland \\} \email{\href{mailto:m.todd@st-andrews.ac.uk }{m.todd@st-andrews.ac.uk }}
\urladdr{\url{https://mtoddm.github.io/}}

\thanks{All authors were partially financed by Portuguese public funds through FCT  -- Funda\c{c}\~ao para a Ci\^encia e a Tecnologia, I.P., in the framework of the projects PTDC/MAT-PUR/4048/2021, 2022.07167.PTDC and UID/00144 - Centro de Matem\'atica da Universidade do Porto.  RC was supported by an FCT scholarship with Reference
Number PD/BD/150456/2019. }

\date{\today}


\begin{abstract}
We consider heavy-tailed observables maximised on a dynamically defined Cantor set and prove convergence of the associated point processes as well as functional limit theorems.  The Cantor structure, and its connection to the dynamics, causes clustering of large observations: this is captured in the `decorations' on our point processes and functional limits, an application of the theory developed in a paper by the latter three authors.
\end{abstract}

\maketitle

\section{Introduction}

Understanding the statistical properties of extreme events in dynamical systems is crucial across various scientific disciplines, notably in climate dynamics where accurate predictions of rare but impactful phenomena are paramount. This paper extends the study of extreme value theory and sums of heavy tailed observables which are maximised on fractal sets, paying particular attention to the clustering profiles of extreme observations.

The motivations to consider fractal maximal sets stem, in particular, from meteorology where the appearance of critical regions with a complex multifractal structure is common, such as in \cite{FMAY17}, where the anomalies for the precipitation frequency data are consistent with an underlying fractal or, in \cite{FAMR19}, where it was observed that the enhancement of greenhouse gases carries the attractor towards some sort of lower dimensional fractal structure. As for the relevance of understanding clustering, we mention its crucial role for assessing compound risks, such in the case of the 2022 European drought addressed in \cite{FPB23}.

The study of rare events for observables with fractal maximal sets is quite recent in the dynamical systems framework (\cite{MP16, FFRS20, FFS21}). In these papers, the main results prove the existence of distributional limits for the partial maxima of stochastic processes arising from uniformly expanding systems with these observables, which, of course, is related to the elapsed time the orbits take to enter small vicinities of the maximal fractal sets (see \cite{FFT10,FFT11}). 

Here, we generalise these results and prove the convergence of decorated point processes of rare events, which provide a more refined approach, aimed at better characterising how extreme events cluster and how their cumulative impact can be quantified. In particular, using the tools developed in \cite{FFT25}, we will also prove the existence of enriched functional limits for sums of heavy tailed observables with such fractal maximal sets.  We note that the theory in \cite{FFT25} is related to ideas in stochastic processes where richer spaces for the limiting process to converge in were developed, such as in \cite{BPS18} (see \cite{FFT25} for a fuller historical account).

We recall that the study sums of heavy tailed stochastic processes is closely intertwined with the extremal properties of the process because the behaviour of the sum is dominated by that of the extreme observations among the summands (\cite{LWZ81}). In particular, in order to prove convergence to a stable law or to a corresponding L\'evy process, the convergence of point processes is usually at the core of the approach (see for example \cite{T10, JPZ20, FFT25,CKM24}, in the dynamical setting, or earlier  works \cite{D83,R87,DH95} in the more classical statistical setting).  

Dynamically generated  heavy tailed stochastic processes can occur either because the observable is unbounded with a polynomial type of spike on the maximal set, or because the dynamics has some some source of non-hyperbolicity, like an indifferent periodic point or a cusp, where the dynamics is severely decelerated so that, after inducing,  there is a piling of observations that aggregate   making the induced observable heavy tailed. To our knowledge, the sources of heavy tailed behaviour in the literature have been taken either finite or at most countable, which is one of the reasons for the interest in considering fractal sets.

In order to be able to give concrete formulae and describe the point processes, particularly their decorations, we will consider uniformly hyperbolic systems (Markov full branched linear maps) and observables maximised on dynamically generated Cantor sets associated to these maps. We remark that although we work with simple models, they capture the essence of the limiting behaviour and we believe that the spirit of our findings should prevail in more general settings, where closed formulas would be impossible to obtain. 

One of the highlights of this paper lies in the construction of an observable function maximised on Cantor sets that still has a continuous distribution function. This is achieved by considering a carefully designed symbolic distance to the maximal set, which is necessary to apply the theory developed in \cite{FFT25} to obtain enriched functional limit theorems and the convergence of decorated point process of rare events. This contrasts with the Cantor ladder type of observables considered in \cite{MP16,FFRS20, FFS21}, which give rise to discrete distributions. 

The formal statements of the main results involve introducing objects and notations that would overload this introduction and for this reason we defer it to Section~\ref{subsec:statement}. Nonetheless, in the following subsections we will give an informal version of these results.

\subsection{Enriched functional limit theorems}

We will consider a finite full-branched interval map and a corresponding dynamically defined Cantor set.  Then, defining an observable maximised on this Cantor set, we obtain a sequence of stationary random variables $X_0, X_1, \ldots$ with an $\alpha$-heavy tail and 
prove a suitable functional limit theorem, namely find limits of 

$$S_n(t)=\sum_{i=0}^{\floor{nt}-1} \dfrac{1}{n^{\frac{1}{\alpha}}}X_i, \ t \in [0,1].$$

We will see that the limit, which will live in the space $F'$ of decorated c\`adl\`ag functions, introduced in \cite{FFT25}, can be seen as an $\alpha$ stable L\'evy process, denoted by $V$, with decorations at each jump, which capture the excursion performed by the finite time process $S_n$, when a cluster of abnormally high observations occurs (note that these correspond to visits of the orbit to close vicinities of the maximal (Cantor) set). In this case, the compatibility between the systems considered and the structure of the maximal (Cantor) set imply that the decoration corresponding to the jumping time, $s$, can be described by a  c\`adl\`ag function of the form:
\begin{equation*}
		e_V^{s}(t)=V(s^{-})+U^{-\frac{1}{\alpha}}\displaystyle\sum_{0 \leq j \leq \floor{\tan(\pi t/2)}} (1-\theta)^{\frac{j}{\alpha}}, \ t \in [0,1]
	\end{equation*}
	where $0<\theta<1$ is the extremal index, which gives the reciprocal of the average number of abnormal observations within a cluster and, in this case, is related to the decay rate of the measure of the iterative approximations of the maximal (Cantor) set, during its construction. In the case of the ternary Cantor set, $\theta=2/3$.

\subsection{Point processes with decorations}

The convergence of enriched functional limit theorems is derived from the convergence of point processes with decorations. Point processes are a very powerful tool to study rare events and, in particular, sums of heavy tailed observations. An illustration of their convenience is the fact that the L\'evy processes obtained here can be written as an integral with respect to a Poisson random measure that underlies the limiting point process.   

The construction of the point processes is based on splitting $n$ observations into $k_n$ blocks of size $r_n$, where  both $(k_n)_n, (r_n)_n$ diverge. Then we consider a point process with a time component keeping record of each block  and a decoration consisting of a vector of increasing length with all block observations conveniently normalised:  
\begin{equation*}
N_n=\sum_{i=1}^{\infty} \delta_{\left(i/k_n,\mathbb{X}_{n,i}\right)},
\end{equation*}
where $\mathbb{X}_{n,i}=\left(n(X_{(i-1)r_n})^{-\alpha},n(X_{(i-1)r_n+1})^{-\alpha},\ldots, n(X_{r_n-1})^{-\alpha}\right)$.

The main result establishes that these point processes after embedded in a proper space converge to a limiting process that can be written as a bidimensional Poisson process with Lebesgue intensity measure decorated with a  
bi-infinite sequence, which in this case is equal to $(1-\theta)^j$ for all $j\geq0 $ and $\infty$ for all $j\leq-1$.

\subsection{Structure of the paper}
In Section~\ref{sec:cantor-set} we outline our class of interval maps and the observables we will use, which are maximised on a dynamically defined Cantor set.  In Section~\ref{sec:FLTs} we explain the theory of functional limit theorems for heavy tailed systems with clustering.  This involved defining sequence spaces $l_\infty^+, \tilde l_\infty^+$ aligned to a `block structure', defining the anchored tail process, defining the dependence conditions $\D_{q_n}^{*}$ and $\D_{q_n}^{'*}$, defining Rare Events Point Processes and explaining what the enriched functional limit theorem in \cite{FFT25} is.  In Section~\ref{sec:application} we apply this theory to our class of examples.  This long section is broken down into the statement of the main theorems, defining our constants and normalisation sequences, the extremal index, checking the dependence requirements, showing the anchored tail process is well defined and finally dealing with the small jumps condition in the case $1<\alpha<2$, using ideas from \cite{CheNicTor24}.

\section{A Fractal Maximal Set}
\label{sec:cantor-set}
We divide $[0, 1]$ into interleaved domains $I_i$ and $J_i$, on each of which we define linear bijections to $[0, 1]$: the former intervals will define our dynamical attractor.

 Let $0\le a_0<a_1<\cdots< a_{2k_I-1}\le 1$ and $I_i= [a_{2(i-1)}, a_{2i-1}]$ for $i=1, \ldots, k_I$.  Next define $J_1, \ldots, J_{k_J}$ be the complementary (open) intervals in order where $k_J\in \{k_I-1, k_I, k_I+1\}$.  So $k_J=k_I-1$ if $a_0=0$ and $a_{2k_I-1}= 1$; $k_J=k_I+1$ if $a_0>0$ and $a_{2k_I-1}< 1$; and $k_J=k_I$ in the other two cases. Let $\hat k:=k_I+ k_J$.  Writing the $I_i$ and $J_i$s together in order, given $1\le k\le \hat k$ it will sometimes be convenient to let $(I,J)_k$ be the $k$-th interval in this list, so either some $I_i$ or some $J_i$.  Moreover, let $\I_I$ be the indices $k$ where $(I,J)_k=I_i$ for some $i$ and  $\I_J$ be the indices $k$ where $(I,J)_k=J_i$ for some $i$ (note $\{1, \ldots, \hat k\} = \I_I\cup \I_J$).

Let $g_i:I_i \to [0,1]$ and $h_j:J_j \to (0,1)$ be linear bijections (making the obvious change if $J_1$ contains 0 and/or $J_{k_J}$ contains 1).  We also assume that $|I_i|, |J_i|\le 1/2$ so that the derivatives of $g_i, h_i$ are greater than or equal to two. 
Let $m_i$ denote the absolute value of the derivative of $g_i$ and let $G:[0,1] \to [0,1]$ be such that the restriction of $G$ to $I_i$ is $g_i$ and to $J_i$ is $h_j$.

The collection of the intervals $I_i$ and $J_j$, with $i=1,\ldots, k_I$ and $j=1,\ldots,k_J$ forms a measurable partition of $[0,1]$ that we denote by $\mathcal P$. Then we define by recursion the sequence of partitions $\mathcal P^n=\bigvee_{j=0}^{n-1}G^{-j}(\mathcal P)$, given by the $n$-th join of $\mathcal P$. The elements of $\mathcal P^n$ are called $n$-cylinders and $n$ will be referred to as the \emph{depth} of the cylinder $\omega\in\mathcal P^n$. Note that, since $G$ is linearly expanding and full-branched, if $\omega_n(x)$ denotes the element of $\mathcal P^n$ that contains $x$ then $G^n(\omega_n(x))=[0,1]$ and $\cap_{n\in\N}\omega_n(x)=\{x\}$.

Let $\Lambda_0=[0,1]$ and, for all $n \in \mathbb{N}$, define
\begin{equation}\label{eq:Lambda_n}
\Lambda_n:=\bigcup_{i=1}^{k_I} g_i^{-1}(\Lambda_{n-1})
\end{equation}
and
\begin{equation}\label{eq:Lambda}
\Lambda:=\displaystyle\bigcap_{n \in \mathbb{N}} \Lambda_n.
\end{equation}

So $\Lambda$ can be thought of as a dynamically defined Cantor set.

A natural example of such a system, is for $k=2$, $I_1=\left[0,\dfrac{1}{3}\right]$, $J_1=\left(\dfrac{1}{3},\dfrac{2}{3}\right)$, $I_2=\left[\dfrac{2}{3},1\right]$ and $g_1(x)=3x$, $h_1(x)= 3x-1$ and $g_2(x) = 3x-2$.  Then $\Lambda$ here is the ternary/middle-$\frac{1}{3}$ Cantor set.

It is always the case that $([0,1],\mathcal{B}_{[0,1]},\m,G)$, where $\m$ denotes the Lebesgue measure on $[0,1]$, is a probability preserving dynamical system.
We will define an observable on this system so that the theory from \cite{FFT25} applies, and whose maximal set is precisely $\Lambda$. For that, we define a distance on $[0,1]$ as follows.

Let $\Sigma= \Sigma_{\hat k}=\{1,\ldots,\hat k\}^{\N}$ and let $\sigma:\Sigma \to \Sigma$ be the left-shift map.
For $x \in [0,1]$, let $\underline{x}=( x_1, x_2 \ldots)\in \Sigma$ be defined by $x_i=k\in \{1, \ldots, \hat k\}$ if $G^{i-1}(x)\in (I, J)_k$. 
Observe that $x \in \Lambda \iff \underline{x}=(x_1, x_2, \ldots)$ where $x_i\in \I_I$ for all $i \in \mathbb{N}$.

\begin{definition}\label{def:cantor-psi}
Let $\underline{x}  =(x_1, x_2, \ldots)\in \Sigma$. Let $\mathcal{I}_x=\{i \in \mathbb{N}: x_i \in \I_J\}$. Let
\begin{equation}\label{eq:cantor-psi}
\psi(x)=\sum_{i \in \mathcal{I}_x} 2^{-i}.
\end{equation}
\end{definition}
$\psi$ provides us with a notion of distance to $\Lambda$ such that the earlier a point exits the dynamical construction of $\Lambda$ (in the sense of (\ref{eq:Lambda_n})) the bigger the corresponding distance to $\Lambda$.  We will also view  $\psi$ is a $[0,1]$-valued random variable defined on $([0,1],\B[0,1],\m)$.

\begin{figure}
\includegraphics[width=14cm]{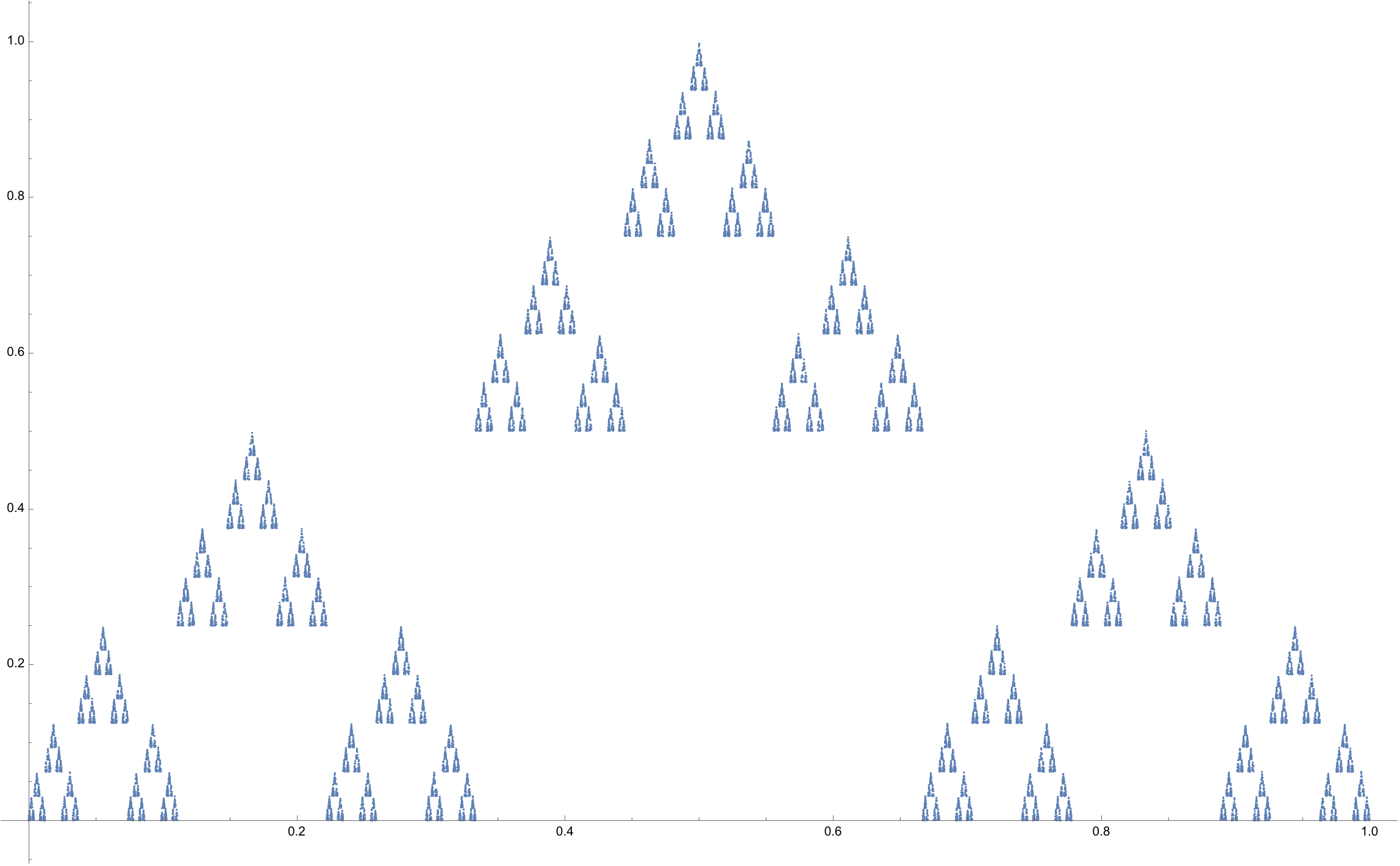}
\caption{Graph of the function $\psi$ for the ternary Cantor set.}
\label{fig:psi}
\end{figure}

\begin{lemma}\label{lem:proper-psi}
\begin{enumerate}[(a)]
\item $\psi(x)=0 \iff x \in \Lambda$;
\item $\psi(x)<2^{-n} \iff x \in \Lambda_n \setminus \{p_1,\dots,p_{k_I^n}\}$, where 
$$\{p_1,\dots,p_{k_I^n}\}=\left\{x \in [0,1]: x_i\in \I_I \text{ for all } i=1, \ldots n, \text{ and } x_i\in \I_J \text{ for } i\ge n+1 \right\};$$
\item for  $F:[0,1] \to [0,1]$ the distribution function of $\psi$, 
\begin{enumerate}[(i)]
\item $F$ is strictly increasing on $[0, 1]$;
\item $F(2^{-n}) = |\Lambda_n| = \lambda^n$ for $\lambda:= \left|\cup_{i=1}^{\I_I}I_i\right|= \sum_{i=1}^{\I_I}\frac1{m_i}$; 
\item we can view $F \circ \psi$ as a uniformly distributed random variable.
\end{enumerate}
\end{enumerate}
\end{lemma}

\begin{proof}
The proofs of (a) and (b) are immediate.  For (c) the strict increasing property follows from the definition of $\psi$ and (ii) follows from (b).  (iii) is a standard fact in probability, which results from observing that
$\mathbb{P}\left(F \circ \psi \leq z\right)=\mathbb{P}\left(\psi \leq F^{-1}(z)\right)=F(F^{-1}(z))=z.$
\end{proof}

\begin{figure}
\includegraphics[width=14cm]{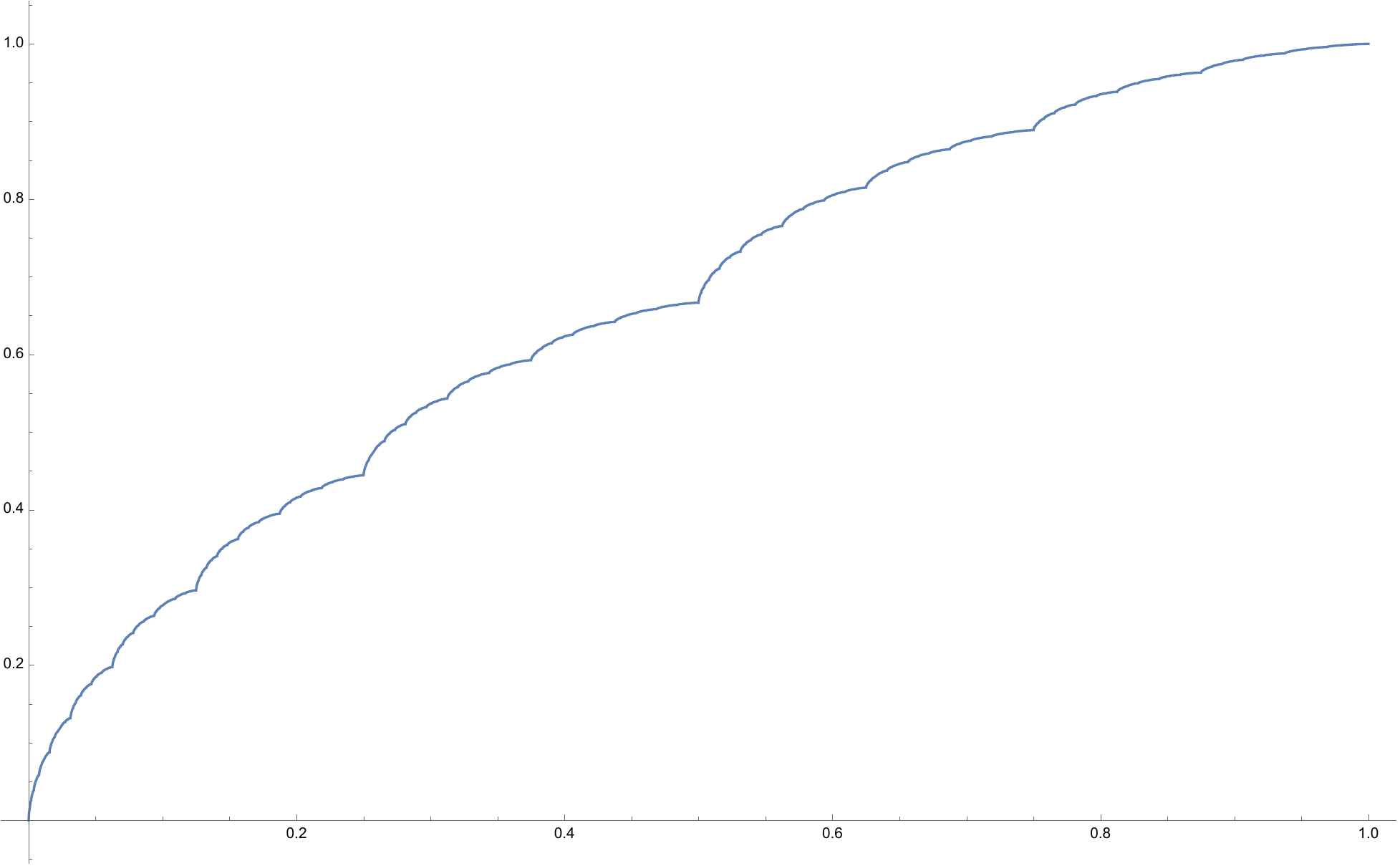}
\caption{Graph of the the distribution function of $\psi$ for the ternary Cantor set.}
\label{fig:F}
\end{figure}

We may now define our observable.

\begin{definition}\label{def:cantor-phi}
For $\alpha \in (0,2)$, let
\begin{equation}\label{eq:cantor-phi}
\varphi_{\alpha}(x):=(F \circ \psi(x))^{-\frac{1}{\alpha}}.
\end{equation}
where $F:[0,1] \to [0,1]$ is the distribution function of $\psi$.  Moreover, define
\begin{equation}\label{eq:cantor-X_n}
X_n:=\varphi_{\alpha} \circ G^n
\end{equation}
for all $n \in \mathbb{N}_0$.  
\end{definition}

Note that since $G$ is $\m$-invariant, $X_n$ is a stationary process.

\begin{lemma}\label{lem:cantor-alpha-reg}
\begin{enumerate}[(a)]
\item $\phi_\alpha(x) = \infty$ iff $x\in \Lambda$;
\item if $x\in \Lambda_{n} \setminus \{p_1,\dots,p_{k_I^n}\}$ and $y\in \Lambda_{n+1}$ then $\phi_\alpha(x)<\phi_\alpha(y)$;
\item $\{X_n>u\} = \{ \psi\circ G^n<F^{-1}(u^{-\alpha})\}$;
\item if $u'>u$ then $\m\{X_n>u'\}<\m\{X_n>u\}$.
\end{enumerate}
\end{lemma}

\begin{proof}
Parts (a), (b) and (c) are immediate from the definition.  Part (d) then follows from part (b) along with  Lemma~\ref{lem:proper-psi}(c).
\end{proof}

\section{General theory of functional limit theorems}
\label{sec:FLTs}

In this section we present a simplified version of the theory in \cite{FFT25} which, to fit with our application outlined above, deals with real non-negative random variables $X_0, X_1, \ldots$ distributed according to $\p$, defined via a dynamical system $f:[0, 1]\to [0, 1]$ with invariant probability measure $\mu$ and an observable $\phi:[0, 1]\to [0, \infty]$ with $X_n= \phi\circ f^{n-1}$ (hence $\p(X>u)= \mu(\phi>u)$, but we will also write this as $\mu(X>u)$).  

We define sequences $(k_n)_{n\in\N}$, $(r_n)_{n\in\N}$, $(t_n)_{n\in\N}$, $(q_n)_n$, which we assume to be such that 
\begin{equation}
\label{eq:kn-sequence}
k_n,r_n,t_n\xrightarrow[n\to\infty]{}\infty\quad \text{ where }  k_n\cdot t_n = o(n), k_n\cdot r_n\sim n \text{ and } q_n=o(r_n).
\end{equation}
As in \cite[(3.2), (3.3)]{FFT25}, these relate to a blocking argument with $k_n$ blocks of size $\lfloor n/k_n\rfloor$, where $t_n$ is a time gap inserted to achieve asymptotic independence of the blocks, while $q_n$ is the clustering run length, \ie all abnormal observations occurring within a time difference of at most $q_n$ units between each other belong to the same cluster. In most dynamical applications $q_n=q$ for all $n\in\N$, such as when the observable is maximised at a single periodic point of period $p$, in which case $q_n=p$ (see \cite{FFT12}). In fact, we will see that, in this paper, for our applications we only require $q_n=1$.
 We assume for each $\tau\ge 0$ there is $(u_n(\tau))_n$ such that
\begin{equation}
\label{eq:un}
\lim_{n\to \infty} n\p(X_0> u_n(\tau))= \tau, \lim_{\tau_1\to 0, \tau_2\to \infty}\p\left(u_n(\tau_1)<X_0< u_n(\tau_2)\right)=1
\end{equation}
and that $u_n:[0, \infty]\to [0, \infty]$ is bijective. 

We say that $X_0$ has \emph{$\alpha$-regularly varying tails} if there is a sequence $(\mathfrak{a}_n)_n$ such that for any $y>0$,
\begin{equation*}
\lim_{n\to\infty}n\p(X_0>y\mathfrak{a}_n)=y^{-\alpha}.
\end{equation*}

Define
$$U_n(\tau):=\{X_0>u_n(\tau)\}.$$

\subsection{Sequence spaces and the anchored tail process}
\label{ssec:seqtail}

For $i< j\in\{0,\ldots,n\}$ we will use the notation 
\begin{equation}
\label{eq:normalised-block}
\X_n^{i, j}:=\left(u_n^{-1}(X_i), \ldots,  u_n^{-1}(X_{j-1})\right),\quad \X_{n,i}:=\X_n^{(i-1)r_n, ir_n}.
\end{equation}

Define
\begin{align*}
l_\infty^+&:=\left\{\mathbf{x}=(x_j)_j\in (0, \infty]^{\Z}\colon\;\lim_{|j|\to\infty}x_j=\infty\right\}.
\end{align*}

Note that we can embed $\cup_{n\in\N} (0, \infty]^n$ 
 into $l_\infty^+$  
 by adding a sequence of $\infty$ before and after the $n$ entrances of any element of  $(0, \infty]^n$ 
  For example, 
$\mathbb X_{n}^{i,j}$ can be seen as an element of $l_\infty^+$ by identifying it with 
\begin{equation*}
\left(\ldots,\infty,\infty,u_n^{-1}(X_{i}),\ldots,{u_n^{-1}(X_{j-1})},\infty,\infty,\ldots\right).
\end{equation*}

We can define the left shift $\TT$ on this space, which moves every element one place to the left.
We define the quotient space $\tilde l_\infty^+=l_\infty^+/{\sim}$ where $\sim$ is the equivalence relation defined by $\mathbf x\sim\mathbf y$ if and only if there exists $k\in\Z$ such that $\TT^k(\mathbf x)=\mathbf y$. Also let $\tilde \pi$ denote the natural projection from $ l_\infty^+$ to $\tilde l_\infty^+$, which assigns to each element $\mathbf x$ of $ l_\infty^+$ the corresponding  equivalence class $\tilde\pi(\mathbf x)=\tilde{\mathbf x}$ in $ \tilde l_\infty^+$. Given any vector $v$ of $(0, \infty]^m$, for some $m\in\N$, we write $\tilde\pi( v)$ for the projection of the natural embedding of $v$ into $ l_\infty^+$ to the quotient space $\tilde l_\infty^+$. Namely,
\begin{equation*}
\tilde\pi(\mathbb X_{n}^{i,j})=\tilde\pi\left(\left(\ldots,\infty,\infty,u_n^{-1}(X_{i}),\ldots,{u_n^{-1}(X_{j-1})},\infty,\infty,\ldots\right)\right).
\end{equation*}
We let $\tilde\infty\in \tilde l_\infty^+$ be the sequence where all terms are $\infty$.

We will assume in the theory, and prove in practice,  the existence of a process $(Y_j)_{j\in\Z}\in l_\infty^+$ satisfying the following assumptions: 
\begin{enumerate}

\item \label{Y-def} 
$\mathcal L\left(\frac1\tau \mathbb X_n^{r_n+s, r_n+t}\;\middle\vert\; X_{r_n}>u_n(\tau) \right)\xrightarrow[n\to\infty]{}\mathcal L\left((Y_j)_{j=s,\dots,t}\right),$ for all $s<t\in\Z$ and all $\tau>0$;

\item \label{spectral-process} the process $(\Theta_j)_{j\in\Z}$ given by $\Theta_j=\frac{Y_j}{Y_0}$ is independent of $Y_0$; 

\item \label{Y-infinity} $\lim_{|j|\to\infty} Y_j=\infty$ a.s.;

\item \label{positive-EI} $\p\left(\inf_{j\leq -1}Y_j\geq 1\right)>0$.

\end{enumerate}

Here $(r_n)_n$ is assumed to satisfy \eqref{eq:kn-sequence}: in our applications it is the sequence appearing in $\D_{q_n}^*$ and $\D^{'*}_{q_n}$ below.  We call this the \emph{transformed tail process}.

We can then define:

\begin{definition}
\label{def:piling-process}
Assuming the existence of a sequence $(Y_j)_{j\in\Z}$ satisfying conditions \eqref{Y-def}--\eqref{positive-EI}, we define the \emph{transformed anchored tail process} $(Z_j)_{j\in\Z}$ as a sequence of random numbers satisfying 
$$
\mathcal L\left((Z_j)_{j\in\Z} \right)=\mathcal L\left((Y_j)_{j\in\Z}\;\middle\vert\; \inf_{j\leq -1}Y_j\geq 1\right).
$$
\end{definition}
The anchor $Z_0$ was chosen so that it marks the beginning of a new cluster.
We consider a polar decomposition of the transformed anchored tail process by defining the random variable $L_Z$ and the process $(Q_j)_{j\in\Z}$ by
\begin{equation}
\label{eq:polar}
L_Z:= \inf_{j
\in\Z}Z_j \qquad Q_j:=\frac{Z_j}{L_Z}.
\end{equation}

\subsection{The  $\D_{q_n}^{*}$ and  $\D_{q_n}^{'*}$ conditions and the extremal index}
\label{ssec:DsEI}

For $B\subset [0 ,1]$, $J$ a subset of $[0, \infty)$ and $q\in\N$ define
$$\W_J(B):= \bigcap_{i\in J\cap\N_0} f^{-i} (B^c), \ \W_J^c(B):= (\W_J(B))^c= \bigcup_{i\in J\cap\N_0} f^{-i} (B),$$
 and
$$B^{(q)}:= B\cap \bigcap_{i=1}^{q}f^{-i}(B^c),\qquad B^{(0)}:=B.$$

Now let $E$ be an element of the ring generated by the half open intervals of $[0,1]$ so that $E=\cup_{\ell=1}^{N}J_\ell=\cup_{\ell=1}^{N}[a_\ell,b_\ell)$ and for each $\ell=1,\ldots,N$, let $A_\ell$ be an element of the ring generated by the half open intervals of $[0,\infty)$ so that $A_\ell=\cup_{s=1}^{n_\ell}[{\tau_s''},\tau'_s)$ for some $n_\ell\in \N$ and some $\tau_1''<\tau_1'<\tau_2''<\cdots <\tau_{n_\ell}''<\tau_{n_\ell}'$.

Given $E$ and $A_\ell$ as above, define
\begin{equation}
\label{def:Anl}
J_{n,\ell}:=k_nr_n J_\ell:=[k_nr_na_\ell,k_nr_nb_\ell)
,\quad A_{n,\ell}:=\bigcup_{s=1}^{n_\ell} \left\{u_n^{-1}(X_0)\in[\tau''_s,\tau'_s)\right\}=\bigcup_{s=1}^{n_\ell}U_n(\tau'_s)\setminus U_n(\tau_s'')
\end{equation}

\textbf{Condition $\D_{q_n}^{*}$}. We say that $\D_{q_n}^{*}$ holds for the sequence ${X}_0,{X}_1,\dots$ if there exist sequences $(k_n)_{n \in \mathbb{N}}, (r_n)_{n \in \mathbb{N}}, (t_n)_{n \in \mathbb{N}}$ and $(q_n)_{n \in \mathbb{N}}$ as defined in \eqref{eq:kn-sequence}, such that for every $N,t,n \in \mathbb{N}$ and every $J_\ell, A_\ell$ as above, for $\ell=1,\dots,N$, we have
\begin{equation*}
	\left\lvert \mathbb{P}\left(A_{n,\ell}^{(q_n)} \cap \bigcap_{i=\ell}^{N} \mathscr{W}_{J_{n,i}}\left(A_{n, i}^{(q_n)} \right)\right) - \mathbb{P}\left(A_{n, \ell}^{(q_n)} \right)\mathbb{P}\left(\bigcap_{i=\ell}^{N} \mathscr{W}_{J_{n,i}}\left(A_{n, i}^{(q_n)} \right)\right) \right\rvert \leq \gamma(n,t_n)
\end{equation*}
where $\min\{J_{n,\ell} \cap \mathbb{N}_0\} \geq t_n+q_n$ and $\lim_{n \to \infty} n\gamma(n,t_n)=0$.

\textbf{Condition $\D_{q_n}^{'*}$}. We say that $\D_{q_n}^{'*}$ holds for the sequence ${X}_0, {X}_1,\dots$ if there exist sequences $(k_n)_{n \in \mathbb{N}}, (r_n)_{n \in \mathbb{N}}, (t_n)_{n \in \mathbb{N}}$ and $(q_n)_{n \in \mathbb{N}}$ as defined in \eqref{eq:kn-sequence}, such that for every $A_1=\cup_{s=1}^{N_1}[{\tau''}_s,\tau'_s)$, 
\begin{equation*}
	\lim_{n \to \infty} n\mathbb{P}\left(A_{n,1}^{(q_n)} \cap \mathscr{W}^c_{[q_n+1,r_n)}\left(A_{n,1} \right)\right)=0.
\end{equation*}
\begin{remark}
\label{rem:D-prime-two}
Note that if $(q_n)_{n\in\N}$ and $(\tilde q_n)_{n\in\N}$ are two sequences satisfying \eqref{eq:kn-sequence} and $q_n\leq\tilde q_n$, for all $n\in \N$.  Then $\D'_{q_n}$ implies $\D'_{\tilde q_n}$. As observed in \cite[Remark~2.11]{AFF20}, one should try first to find the smallest constant sequence $q_n=q$ that makes $\D_{q_n}^{'*}$ hold,
establishing, in this way, the run length. We will see that here we can actually take $q_n=1$ for all $n\in\N$.
\end{remark}

If, for all $\tau>0$, the following limit exists:
\begin{equation}
\label{eq:EI}
\theta=\lim_{n\to\infty}\frac{\p(U_n^{(q_n)}(\tau))}{\p(U_n(\tau))},
\end{equation}
then we call $\theta$ the \emph{extremal index}.

\subsection{Complete convergence of the Rare Events Point Processes}\label{subsec:cantor-repp}
As mentioned above, at the core of all convergence results stated and mentioned in this paper is the convergence of the so-called Rare Events Point Processes (REPP), which are spatio-temporal clustering processes that keep record of the cluster profiles and their time occurrences:
\begin{equation}
\label{def:Nn}
N_n=\sum_{i=1}^{\infty} \delta_{\left(i/k_n,\tilde{\pi}(\mathbb{X}_{n,i})\right)}.
\end{equation}

The complete convergence of these REPP follows from \cite[Theorem 3.18]{FFT25}, which assuming that the transformed anchored tail process exists and is well defined (see Definition~\ref{def:piling-process}) and that conditions $\D_{q_n}^{*}$ and  $\D_{q_n}^{*'}$ are satisfied, states that $N_n$ given in \eqref{def:Nn} converges weakly in the space of boundedly finite point measures on $\mathbb{R}_0^{+} \times \tilde{l}_{\infty}^+ \setminus \{\tilde{\infty}\}$ with weak$^{\#}$ topology (see \cite[Appendix~C]{FFT25} for details) to
the Poisson process with decorations, which can be written as:
\begin{equation}
\label{def:N}
N=\sum_{i=1}^{\infty} \delta_{(T_i,U_i\mathbf{\tilde{Q}}_i)},
\end{equation}
where $T_i$ and $U_i$ are independently (of $i$ and of each other) distributed according to $Leb$ and $\theta Leb$ respectively,  for $\theta$ as in \eqref{eq:EI}, and $(\mathbf{\tilde{Q}}_i)_{i\in\N}$ is an i.i.d sequence such that each $\mathbf{\tilde{Q}}_{i}\in\tilde l_\infty^+ \setminus \{\tilde{\infty}\}$ has the same distribution as $\tilde\pi((Q_j)_{j\in\Z})$, where the bi-infinite sequence $(Q_j)_j$ is as in \eqref{eq:polar}.

\subsection{Enriched functional limit theorem}\label{cantor-flt}
For processes $X_0, X_1, \ldots$ with $\alpha$-regularly varying tails, for which we can prove the convergence of the REPP defined above, applying \cite[Theorem 2.5]{FFT25}, we obtain an enriched functional limit theorem (FLT) for the normalised sums: 
\begin{equation*}
		S_n(t)=\sum_{i=0}^{\floor{nt}-1} \dfrac{1}{\mathfrak{a}_n}X_i-tc_n, \ t \in [0,1],
	\end{equation*}
where the sequence $(c_n)_{n\in\N}$ is such that $c_n=0$ if $0<\alpha<1$ and 
\begin{equation}
\label{eq:centering}
c_n=\frac{n}{\mathfrak{a}_n} \E\left(X_0\Ind_{|X_0|\leq \mathfrak{a}_n}\right), \qquad \text{for } 1<\alpha<2.
\end{equation}
In order to guarantee the existence of a limit for the c\`ad\`ag function $S_n(t)$, which sometimes is not possible with any of the Skorohod metrics in $D([0,1])$, the space of  c\`ad\`ag functions on $[0,1]$, and to keep the information regarding the excursions performed by finite time process during a cluster of abnormal observations that gets collapsed on the same discontinuity of the limit process, in \cite{FFT25}, the authors introduced the enriched space $F'([0,1])$ as the space of triplets $(v, S^v,\{e_v^s\}_{s\in S^v})$, where $v$ is the main  c\`ad\`ag function, $S^v$ a set containing all its discontinuities and $e_v^s$ a decorating excursion associated to each element of $S^v$. 
This excursion is an element of  $\tilde D([0,1])$, the equivalence space of  c\`ad\`ag functions on $[0,1]$, where two elements are identified if there is a strictly continuous time reparametrisation that sends one into another. We refer to \cite[Section~2.3.1]{FFT25} for the details about this space, the corresponding metric and notions of convergence there. 

From  \cite[Theorem 2.5]{FFT25}, in the case $0<\alpha<1$, we obtain that the c\`ad\`ag function $S_n(t)$ embedded in the space $F'$ converges to  $(V,\text{disc}(V),(e_V^s)_{s \in \text{disc}(V)})$, where $V$ is an $\alpha$-stable Lévy process on $[0,1]$, which can be written as:
	\begin{equation}
	\label{eq:V-01}
		V(t)=\sum_{T_i \leq t}\sum_{j \in \mathbb{Z}} U_i^{-\frac{1}{\alpha}}\mathcal{Q}_{i,j}
	\end{equation}
	and the excursions are given by
	\begin{equation}
	\label{eq:excursion}
		e_V^{T_i}(t)=V(T_i^{-})+U_i^{-\frac{1}{\alpha}}\displaystyle\sum_{j \leq \floor{\tan(\pi(t-\frac{1}{2}))}} \mathcal{Q}_{i,j}, \ t \in [0,1]
	\end{equation}
	where $\mathcal{Q}_{i,j}=\mathbf{\tilde{Q}}_{i,j}^{-\frac1\alpha}$ (zero if $\mathbf{\tilde{Q}}_{i,j}=\infty$), with $\mathbf{\tilde{Q}}_i=\left(\mathbf{\tilde{Q}}_{i,j}\right)_{j}$, $T_i$ and $U_i$ as in $N$ in \eqref{def:N}.

In the case $1<\alpha<2$, we need the next two extra conditions to hold. Namely, for all $\delta>0$
\begin{equation}
\label{eq:small-jumps}
\lim_{\varepsilon\to0}\limsup_{n\to\infty}\p\left(\max_{1\leq k\leq n}\left\|\sum_{j=1}^k\left(X_j\Ind_{|X_j|\leq \varepsilon a_n}\right)-\E\left(X_j\Ind_{|X_j|\leq \varepsilon a_n}\right)  \right\|\geq \delta a_n\right)=0
\end{equation}
and, moreover,
\begin{equation}
\label{eq:Q-condition->1}
\E\left(\left(\sum_{j\in\Z}\|\mathcal{ Q}_j)\|\right)^\alpha\right)<\infty.
\end{equation}
Then, under these two assumptions as well, $S_n(t)$  converges to  $(V,\text{disc}(V),(e_V^s)_{s \in \text{disc}(V)})$ in $F'$, where $V$ has the more complicated expression:
\begin{align}
V(t)&=\lim_{\varepsilon\to 0}\Bigg(\sum_{T_i\leq t}\sum_{j\in\Z} U_{i}^{-\frac1\alpha}\mathcal{Q}_{i,j}\Ind_{\{|U_{i}^{-\frac1\alpha}\mathcal{Q}_{i,j}|>\varepsilon\}}\nonumber\\
&\hspace{4.5cm}
-t\theta\int_{0}^{+\infty}\E\Bigg(y\sum_{j\in\Z}\mathcal{Q}_{j}\Ind_{\{\eps< y|\mathcal{Q}_{j}|\leq 1\}}\Bigg)d(- y^{-\alpha}) \Bigg),
\label{eq:V-12}
\end{align}
while $e_V^{T_i}(t)$ is again as in \eqref{eq:excursion}.

\section{Application of the theory to our examples}
\label{sec:application}

We apply the theory described above to the setting described in Section~\ref{sec:cantor-set}. In particular, we have now all the tools and notation introduced in order to state our main results.

\subsection{Statement of main results}
\label{subsec:statement}

We begin with the convergence of the REPP.
\begin{theorem}
\label{thm:REPP-conv}
Let $G$ be as defined in Section \ref{sec:cantor-set}, let $\varphi_{\alpha}$ be as in Definition \ref{def:cantor-phi} and $X_0,X_1,\ldots$, as in \eqref{eq:cantor-X_n}. We consider the REPP
$$
N_n=\sum_{i=1}^{\infty} \delta_{\left(\frac{i}{k_n},\tilde{\pi}\left(\left(n(X_{(i-1)r_n})^{-\alpha},n(X_{(i-1)r_n+1})^{-\alpha},\ldots, n(X_{ir_n-1})^{-\alpha}\right)\right)\right)},
$$
where the sequences $(k_n)_n, (r_n)_n$ are as described in \eqref{eq:kn-sequence}.

Then $N_n$ converges in the  weak$^{\#}$  topology to a process $N$ which can be written as in \eqref{def:N}, where $Q_j=(1-\theta)^{-j}$ for all $j\geq 0$, $Q_j=\infty$ for all $j\leq -1$ and
$$
\theta=1-\displaystyle\sum_{i=1}^{\I_I}\dfrac{1}{m_i}= 1-\lambda.
$$
\end{theorem}

From the convergence of the REPP and \cite[Theorem 2.5]{FFT25}, we obtain the following enriched functional limit theorem.
\begin{theorem}
\label{thm:enriched-conv}
Under the same assumptions of the previous theorem, we consider the heavy tailed sums:
$$
S_n(t)=\sum_{i=0}^{\floor{nt}-1} \dfrac{1}{n^{\frac{1}{\alpha}}}X_i-tc_n, \ t \in [0,1],
$$
where $c_n=0$, when $0<\alpha<1$ and is given by \eqref{eq:centering}, when $1<\alpha<2$.

Then, the process $S_n(t)$ converges in $F'$ to $(V,\text{disc}(V),(e_V^s)_{s \in \text{disc}(V)})$, where $V$ is as described in \eqref{eq:V-01}, when $0<\alpha<1$, or as in \eqref{eq:V-12}, when $1<\alpha<2$, and the excursions can be written as
\begin{equation*}
	e_V^{T_i}(t)=V(T_i^{-})+U_i^{-\frac{1}{\alpha}}\sum_{0 \leq j \leq \floor{\tan(\pi t/2)}} (1-\theta)^{\frac{j}{\alpha}}, \ t \in [0,1].
	\end{equation*}
\end{theorem}

\begin{remark}
In light of \cite{FFRS20}, for $\mathcal{M}$ equal to the middle-$\frac{1}{3}$ Cantor set then $\theta<1$ if and only if, among the class of maps $mx \mod 1$, $m \geq 2$, we restrict to $m=3^k$, $k \in \mathbb{N}$. Hence, an extremal index strictly less than $1$ is obtained if and only if the dynamics is somehow compatible with the construction of the middle-$\frac{1}{3}$ Cantor set. We proceed with the computation of the extremal index in the greater generality of $\Lambda$ dynamically defined by $G$ and for our observable $\varphi_{\alpha}$ which is different from the one used in \cite{FFRS20}.
\end{remark}

The rest of this section is dedicated to the proof of Theorem~\ref{thm:REPP-conv}, which essentially means checking that we have the right normalisations, that all conditions of \cite[Theorem 3.18]{FFT25}, \cite[Theorem 2.5]{FFT25} described in Sections \ref{subsec:cantor-repp} and \ref{cantor-flt} hold, and computing the extremal index and the transformed anchored tail process.

\subsection{Tuning constants and normalising sequences}
\begin{lemma}
 $X_n$ have $\alpha$-regularly varying tails.
 \end{lemma}

\begin{proof}
We wish to find $(\mathfrak{a}_n)_{n \in \mathbb{N}}$ such that
$\lim_{n \to \infty}n\mathbb{P}(X_0>y\mathfrak{a}_n)=y^{-\alpha}.$  Since from Lemma~\ref{lem:cantor-alpha-reg},
$$\{X_0>y\mathfrak{a}_n\}=\{F\circ\psi<y^{-\alpha}\mathfrak{a}_n^{-\alpha}\},
$$
it follows that
$$\mathbb{P}(X_0>y\mathfrak{a}_n)=\mathbb{P}(F \circ \psi<y^{-\alpha}\mathfrak{a}_n^{-\alpha})=y^{-\alpha}\mathfrak{a}_n^{-\alpha}.$$
Hence we conclude that $\mathfrak{a}_n=n^{\frac{1}{\alpha}}$ for all $n \in \mathbb{N}$.
\end{proof}

In particular, this lemma means for each $\tau>0$ we can define $(u_n(\tau))_n$ so that 
\begin{equation}
\lim_{n \to \infty}n\mathbb{P}(X_0>u_n(\tau))=\tau, \text{ so } u_n(\tau) = \left(\frac n\tau\right)^{\frac1\alpha}\quad \text{and}\quad u_n^{-1}(z)=nz^{-\alpha}.
\label{eq:unexplicit}
\end{equation}

We next give a geometric description of this set.
By Lemma~\ref{lem:proper-psi} and \eqref{eq:unexplicit} there are $a_i=a_i(n,\tau) \in \{0,1\}$ so that,
\begin{equation*}
F^{-1}((u_n(\tau))^{-\alpha})= F^{-1}\left(\frac\tau n\right)=\sum_{i \in \mathbb{N}} a_i2^{-i}=\sum_{j \in \mathcal{J}_{n,\tau}} 2^{-j}=:u_{n,\tau}
\end{equation*}
where $\mathcal{J}_{n,\tau}=\{i \in \mathbb{N}: a_i=1\}$.

Let 
\begin{equation}
j_{n,\tau}=j_{n,\tau}^{1}:=\min\{j: j \in \mathcal{J}_{n,\tau}\}\;\text{and}\; j_{n,\tau}^{\ell+1}=\min\{j> j_{n,\tau}^{\ell}\colon j \in \mathcal{J}_{n,\tau}\},\;\text{for all}\; \ell\in\N.
\label{eq:j_n}
\end{equation}

If, for some $\ell > 1$, we have $a_i = 0$ for all $i \geq j_{n, r}^\ell$, then we set $ j_{n, r}^s =\infty$ for all $s\ge \ell$.
Note that Lemma~\ref{lem:proper-psi}(c)(ii) implies that $j_{n,\tau}$ grows logarithmically in $n$ (\ie, $\limsup_{n\to\infty}\frac{j_{n,\tau}}{\log n}<\infty$).
Moreover, by Lemma~\ref{lem:cantor-alpha-reg},
\begin{equation}\label{eq:cantor-U_n}
\begin{split}
U_n(\tau)=\{X_0>u_n(\tau)\}=\{\psi<u_{n,\tau}\}=\Lambda_{j_{n,\tau}} \cup H^{(j_{n,\tau})}
\end{split}
\end{equation}
where $H^{(j_{n,\tau})}$ denotes a subset of $\Lambda_{j_{n,\tau}-1} \setminus \Lambda_{j_{n,\tau}}$. 

\begin{remark}
\label{rem:Hjn}
Recall  Lemma~\ref{lem:proper-psi}(b) to justify the appearance of $\Lambda_{j_{n,\tau}}$. Observe that $u_{n,\tau}=2^{-j_{n,\tau}}+\delta$ with $\delta \leq 2^{-j_{n,\tau}}$, so, in particular, $H^{(j_{n,\tau})}=\emptyset$ if $\delta=0$ (which means that $ j_{n,\tau}^{\ell}=\infty$ for all $\ell\geq2$) and $H^{(j_{n,\tau})}=\Lambda_{j_{n,\tau}-1} \setminus (\Lambda_{j_{n,\tau}}\cup  \{p_1,\dots,p_{k_I^{j_{n,\tau}}}\})$, if $\delta=2^{-j_{n,\tau}}$ and where the $p_i$s are as in Lemma~\ref{lem:proper-psi}. Now, $H^{(j_{n,\tau})}$ can be further decomposed into
$$H^{(j_{n,\tau})}=\Lambda'_{j_{n,\tau}^2-j_{n,\tau}}\cup H^{(j_{n,\tau}^2)},$$ 
where $\Lambda'_{j_{n,\tau}^2-j_{n,\tau}}$ is homothetic to $\Lambda_{j_{n,\tau}^2-j_{n,\tau}}$, which was rescaled to fit the corresponding connected components of $\Lambda_{j_{n,\tau}-1} \setminus \Lambda_{j_{n,\tau}}$, and $H^{(j_{n,\tau}^2)}$ is a subset of $\Lambda'_{j_{n,\tau}^2-j_{n,\tau}}\sm\Lambda'_{j_{n,\tau}^2-j_{n,\tau}-1}$. Note that $\Lambda'_{j_{n,\tau}^2-j_{n,\tau}}$ can be written as a union of $j_{n,\tau}^2$-cylinders. This procedure can be repeated to obtain approximations of  $H^{(j_{n,\tau})}$ by unions of cylinders of greater and greater depth.

\end{remark}

\subsection{Extremal Index}\label{sec:cantor-EI}

Recall that
\begin{equation*}
\begin{split}
U_n^{(q)}(\tau)&=\{X_0(x)>u_n(\tau),X_1 \leq u_n(\tau),\dots,X_{q} \leq u_n(\tau)\}\\
&=\{\psi<u_{n,\tau},\psi\circ G \geq u_{n,\tau},\dots,\psi\circ G^{q} \geq u_{n,\tau}\}.
\end{split}
\end{equation*}
Observe that  for $n$ sufficiently large, we have
\begin{equation}\label{eq:cantor-U_n-q_n}
\begin{split}
U_n^{(1)}(\tau)&=\{ \psi<u_{n,\tau},\psi\circ G \geq u_{n,\tau}\}\\
&=\left(\Lambda_{j_{n,\tau}} \cup H^{(j_{n,\tau})}\right) \cap G^{-1}([0,1] \setminus (\Lambda_{j_{n,\tau}} \cup H^{(j_{n,\tau})}))\\
&=H^{(j_{n,\tau})} \cup (\Lambda_{j_{n,\tau}} \cap G^{-1}((\Lambda_{j_{n,\tau}-1} \setminus \Lambda_{j_{n,\tau}}) \setminus H^{(j_{n,\tau})}))
\end{split}
\end{equation}
Now, observing that the points of $U_n^{(1)}(\tau)$ are sent by $G$ into the holes of smaller and smaller generation in the construction of $\Lambda$, it follows that $U_n^{(1)}(\tau)=U_n^{(2)}(\tau)=U_n^{(3)}(\tau)=\ldots$. In light of Remark~\ref{rem:D-prime-two}, this means that $q_n=1$ is then a natural choice to prove condition $\D_{q_n}^{'*}$, which is indeed established in Section~\ref{subsec:dependence-conditions}.


The extremal index $\theta$ can be written as
\begin{equation*}
\theta=\lim_{n \to \infty} \dfrac{\m(U_n^{(1)}(\tau))}{\m(U_n(\tau))}
\end{equation*}
provided the limit exists.

\begin{proposition}\label{prop:cantor-EI}
For $G$ as defined in Section \ref{sec:cantor-set} and observable $\varphi_{\alpha}$ as given in Definition \ref{def:cantor-phi}, the extremal index is $\theta=1-\displaystyle\sum_{i=1}^{\I_I}\dfrac{1}{m_i} = 1-\lambda$. 
\end{proposition}
\begin{proof}
Let $u_{n,\tau} \in [2^{-j},2^{-j+1})$. Then, $j_{n,\tau}=j$ and, by (\ref{eq:cantor-U_n}), $U_n(\tau)=\Lambda_j \cup H^{(j)}$ so that
\begin{equation}\label{eq:cantor-mu-U_n}
\m(U_n(\tau))=\m(\Lambda_j)+\kappa \m(\Lambda_{j-1} \setminus \Lambda_j)
\end{equation}
where $\kappa \in [0,1]$ since $H^{(j)}$ is a fraction of $\Lambda_{j-1} \setminus \Lambda_j$.

Now, by (\ref{eq:cantor-U_n-q_n}),
\begin{equation*}\label{eq:cantor-U_n-q_n-2}
U_n^{(1)}(\tau)=H^{(j)} \cup (\Lambda_j \cap G^{-1}((\Lambda_{j-1} \setminus \Lambda_j) \setminus H^{(j)}))
\end{equation*}
and we have
\begin{equation}\label{eq:cantor-mu-U_n-q_n}
\m(U_n^{(1)}(\tau))=\kappa \m(\Lambda_{j-1} \setminus \Lambda_j)+(1-\kappa)\m(\Lambda_j \setminus \Lambda_{j+1})
\end{equation}
where $\kappa$ is the same as in \eqref{eq:cantor-mu-U_n} (in particular, observe that the second summand reflects the fact that only a fraction of $\Lambda_j \setminus \Lambda_{j+1}$, whose measure is $(1-\kappa)\m(\Lambda_j \setminus \Lambda_{j+1})$, is iterated in one time step to the complement of $H^{(j)}$ in $\Lambda_{j-1} \setminus \Lambda_j$).

\eqref{eq:cantor-mu-U_n} can be rewritten
\begin{equation*}
\m(U_n(\tau))=\left(\sum_{i=1}^{\I_I}\dfrac{1}{m_i}\right) \m(\Lambda_{j-1})+ \kappa\left(1-\sum_{i=1}^{\I_I}\dfrac{1}{m_i}\right) \m(\Lambda_{j-1})=\lambda \m(\Lambda_{j-1})+ \kappa\left(1-\lambda\right) \m(\Lambda_{j-1})
\end{equation*}

and \eqref{eq:cantor-mu-U_n-q_n} can be rewritten
\begin{equation*}
\m(U_n^{(1)}(\tau))=\kappa\left(1-\lambda\right) \m(\Lambda_{j-1}) + (1-\kappa) \left(1-\lambda\right)\lambda\m(\Lambda_{j-1})
\end{equation*}

Finally,
\begin{equation*}
\begin{split}
\theta_n=\dfrac{\m(U_n^{(1)}(\tau))}{\m(U_n(\tau))}&=\dfrac{\kappa\left(1-\displaystyle\lambda\right)+(1-\kappa) \left(1-\lambda\right)\lambda}{\lambda+\kappa\left(1-\lambda\right)}=\dfrac{\kappa \left(1-\lambda\right) \left(1-\lambda\right)+\left(1-\lambda\right)\lambda}{\lambda+\kappa \left(1-\lambda\right)}\\
&=\left(1-\lambda\right)\dfrac{\kappa \left(1-\lambda\right)+\lambda}{\lambda+\kappa \left(1-\displaystyle\lambda\right)}=1-\displaystyle\lambda
\end{split}
\end{equation*}

Therefore, $\theta=\displaystyle\lim_{n \to \infty} \theta_n=1-\displaystyle\lambda$.
\end{proof}

\begin{remark}
Observe that with $\Lambda$ equal to the middle-$\frac{1}{3}$ Cantor set as described in Section~\ref{sec:cantor-set} we obtain $\theta=\dfrac{1}{3}$.
\end{remark}

\subsection{Dependence requirements}\label{subsec:dependence-conditions}
In order to prove that the conditions $\D_{q_n}$ and $\D'_{q_n}$ from Section~\ref{ssec:DsEI} hold in our setting we will approximate the sets $A_{n,\ell}^{(q_n)}$  by a union of cylinders of controlled depth and then use the excellent mixing properties on cylinders  in the spirit of \cite{HP14,DHY21,BF24}, for example.

Since $A_{n,\ell}^{(q_n)}$ is itself a finite union of sets of the form  $U_n(\tau', \tau''):=U_n(\tau')\setminus U_n(\tau'')$, for $\tau''<\tau'$, we restrict the analysis to these building blocks in order to simplify the already heavy notation.

Since
\begin{equation}\label{eq:cantor-A_n,l}
	\begin{split}
		U_n(\tau', \tau'')&=\left(\Lambda_{j_{n,\tau'}} \cup H^{(j_{n,\tau'})}\right) \setminus \left(\Lambda_{j_{n,\tau''}} \cup H^{(j_{n,\tau''})}\right)\\
		&=H^{(j_{n,\tau'})} \cup \left(\Lambda_{j_{n,\tau'}} \setminus \left(\Lambda_{j_{n,\tau''}} \cup H^{(j_{n,\tau''})}\right)\right),
	\end{split}
\end{equation}
$U_n(\tau', \tau'')$ is made up of holes in the dynamical construction of $\Lambda$ (\textit{i.e.} subsets of the form $\Lambda_{k} \setminus \Lambda_{k+1}$ where $k \in \mathbb{N}$), where $j_{n,\tau'}$ is as in \eqref{eq:j_n}.
Moreover, as observed earlier, for $n$ sufficiently large,
\begin{equation}\label{eq:cantor-A_n,l-q_n}
	\begin{split}
		U_n(\tau', \tau'')^{(q)}&=U_n(\tau', \tau'') \cap G^{-1}(U_n(\tau', \tau''))^c \cap \dots \cap G^{-q}(U_n(\tau', \tau''))^c\\
		&=U_n(\tau', \tau'') \cap G^{-1}(U_n(\tau', \tau''))^c
	\end{split}
\end{equation}
so, in particular, taking $q_n=1$ for all $n \in \mathbb{N}$ is the natural choice to prove condition $\D_{q_n}^{'*}$, which is established in  \ref{subsec:dependence-conditions}.

Combining \eqref{eq:cantor-A_n,l} and \eqref{eq:cantor-A_n,l-q_n}, we obtain 
\begin{equation*}\label{eq:cantor-A_n,l-q_n-2}
	U_n(\tau', \tau'')^{(1)}=H^{(j_{n,\tau'})} \ \bigcup \left\{\left((\Lambda_{j_{n,\tau'}} \setminus \Lambda_{j_{n,\tau''}}) \setminus H^{(j_{n,\tau''})}\right) \cap G^{-1}\left((\Lambda_{j_{n,\tau'}-1} \setminus \Lambda_{j_{n,\tau'}}) \setminus H^{(j_{n,\tau'})}\right)\right\}.
\end{equation*}

For a Borel set $A\subset [0,1]$ and $j \in \mathbb{N}$, we define 
	\[
	\Upsilon^+(A, j) := \bigcup_{\omega \in \mathcal P_{j} \colon \omega \cap A \neq \emptyset} \omega \qquad\text{and}\qquad \Upsilon^-(A,j) := \bigcup_{\omega \in \mathcal P_{j} \colon \omega \subset A} \omega.
	\]
	They are the approximations of $A$ from above and below by $j$-cylinders.

\begin{lemma}\label{lem:cantor-cyl-approx}
Let $U_n(\tau', \tau'')^{(1)}$ be given and let $j_{n,\tau'}$ be associated to this set as in \eqref{eq:j_n}. Let 
	$\Gamma_n^{+}=\Upsilon^+(U_n(\tau', \tau'')^{(1)}, 2j_{n,\tau'})$ and $\Gamma_n^{-}=\Upsilon^-(U_n(\tau', \tau'')^{(1)}, 2j_{n,\tau'})$ be the respective outer and inner approximations by unions of $(2j_{n,\tau'})$-cylinders. Then there exists $(\rho_n)_n$ such that
	\begin{equation}\label{eq:cantor-cyl-approx}
		\dfrac{\m(\Gamma_n^{+} \setminus \Gamma_n^{-})}{\m\left(U_n(\tau', \tau'')^{(1)}\right)} \leq \rho_n
	\end{equation}
	where $\displaystyle\lim_{n \to \infty} \rho_n=0$.
\end{lemma}
\begin{proof}
	We write
	\begin{equation*}\label{eq:A_n,l-1-approx}
		U_n(\tau', \tau'')^{(1)}=H^{(j_{n,\tau'})} \cup K
	\end{equation*}
	where $K \subset \Lambda_{j_{n,\tau'}} \setminus \Lambda_{j_{n,\tau'}+1}$. Therefore, 
	\begin{equation}
	\label{cyl-estimate1}
	\begin{split}
	\m\left(U_n(\tau', \tau'')^{(1)}\right)&=\beta \m(\Lambda_{j_{n,\tau'}-1} \setminus \Lambda_{j_{n,\tau'}})+\gamma \m(\Lambda_{j_{n,\tau'}} \setminus \Lambda_{j_{n,\tau'}+1})\\
	&=\beta \left(1-\lambda\right) \m(\Lambda_{j_{n,\tau'}-1}) + \gamma \left(1-\lambda\right) \m(\Lambda_{j_{n,\tau'}})\\
	&=\beta \left(1-\lambda\right) \m(\Lambda_{j_{n,\tau'}-1}) + \gamma \left(1-\lambda\right)\lambda\m(\Lambda_{j_{n,\tau'}-1})\\
	&=\left( \beta \left(1-\lambda\right) + \gamma \left(1-\lambda\right)\lambda \right) \m(\Lambda_{j_{n,\tau'}-1})\\
	&=\left(\beta\theta+\gamma\theta(1-\theta)\right) \lambda^{j_{n,\tau'}-1}
	\end{split}
	\end{equation}
	where $\beta,\gamma \in [0,1]$ and $\theta$ is as in Proposition \ref{prop:cantor-EI} (as $H^{(j_{n,\tau'})}$ is a fraction of $\Lambda_{j_{n,\tau'}-1} \setminus \Lambda_{j_{n,\tau'}}$ while $K$ is a fraction of $\Lambda_{j_{n,\tau'}} \setminus \Lambda_{j_{n,\tau'}+1}$). Observe that the larger the set $H^{(j_{n,\tau'})}$ is, the smaller the set $K$ is and vice-versa. This means that, in particular, $\beta$ and $\gamma$ cannot be simultaneously equal to $0$ or to $1$. 
	
	On the other hand, taking into consideration the definition of the function $\psi$ and of the approximating cylinders, as well as the fact that $G$ is a full branched Markov linear map (which, in particular, means that there exist $\tilde C>0$ and $0<\eta<1$ such that $\m(\omega)\leq \tilde C\eta^n$, for all $n\in\N$ and $\omega\in\mathcal P^n$), we claim that 
	\begin{equation}
	\label{cyl-estimate2}
	\m(\Gamma_n^+\setminus \Gamma_n^-)\leq C\eta^{2j_{n,\tau'}},\quad\text{for some}\quad C>0.
	\end{equation}
	 To see this, we start by noting that since $K$ is a preimage of the complement of $H^{(j_{n,\tau'})}$ in $\Lambda_{j_{n,\tau'}-1} \setminus \Lambda_{j_{n,\tau'}}$, we only need to check that $\m\left(\Upsilon^+(H^{(j_{n,\tau'})},2j_{n,\tau'})\setminus\Upsilon^-(H^{(j_{n,\tau'})},2j_{n,\tau'})\right)\leq C\eta^{2j_{n,\tau'}}$ for some $C>0$. 
	 
	 Recalling Remark~\ref{rem:Hjn}, assume first that $j_{n,\tau'}^2>2j_{n,\tau'}$. This means that $H^{(j_{n,\tau'})}\subset \Lambda_{j_{n,\tau'}}'$, where $\Lambda_{j_{n,\tau'}}'$ is homothetic to $\Lambda_{j_{n,\tau'}}$, which was scaled to fit in one the the holes of $\Lambda_{j_{n,\tau'}}$ that comprise $\Lambda_{j_{n,\tau'}-1}\setminus \Lambda_{j_{n,\tau'}}$ in order to give rise to $\Lambda_{j_{n,\tau'}}'$. Hence, since $\emptyset\subset\Upsilon^-(H^{(j_{n,\tau'})}\subset \Upsilon^+(H^{(j_{n,\tau'})}\subset \Lambda_{j_{n,\tau'}}'$, we have that $$\m\left(\Upsilon^+(H^{(j_{n,\tau'})},2j_{n,\tau'})\setminus\Upsilon^-(H^{(j_{n,\tau'})},2j_{n,\tau'})\right)\leq \m( \Lambda_{j_{n,\tau'}}')\leq C\eta^{2j_{n,\tau'}},$$ for some $C>0$. 
	
	 Now, if $j_{n,\tau'}^2\leq 2j_{n,\tau'}$, we decompose the set $H^{(j_{n,\tau'})}=\Lambda'_{j_{n,\tau}^2-j_{n,\tau}}\cup H^{(j_{n,\tau}^2)}$ further, as in Remark~\ref{rem:Hjn}. Then, observe that $\Lambda'_{j_{n,\tau}^2-j_{n,\tau}}$ can clearly be written as a union of cylinders of $\mathcal P^{2j_{n,\tau'}}$, which means that we only need to estimate the measure of $\Upsilon^+(H^{(j_{n,\tau}^2)},2j_{n,\tau'})\setminus\Upsilon^-(H^{(j_{n,\tau}^2)},2j_{n,\tau'})$. We now consider the case $j_{n,\tau'}^3\leq 2j_{n,\tau'}$, which will lead us to the same conclusion, as before, and otherwise proceed to another decomposition of $H^{(j_{n,\tau}^2)}$ until we end by exhaustion.

 The result now follows from gathering estimates \eqref{cyl-estimate1} and \eqref{cyl-estimate2}.	
\end{proof}

It is a direct consequence of Lemma \ref{lem:cantor-cyl-approx} that the REPP counting the number of hits to $U_n(\tau', \tau'')^{(1)}$ converges if and only if the REPP counting the number of hits to either $\Gamma_n^{+}$ or $\Gamma_n^{-}$ converges, and the limits are the same (see \cite[Lemma~3.13]{BF24}, for example).

As a result, we only need to check that the conditions $\D_{q_n}^*$ and $\D^{'*}_{q_n}$ are verified for the cylinder approximating sets. To this end, we next define a strong form of mixing, which is a property of the maps we are considering here.

\begin{definition}\label{def:phi-mixing} Let $(\mathcal{X},\mathcal{B}_\mathcal{X},\mu,f)$ be a probability preserving system. Let $\mathcal{F}_{0,k}$ denote the $\sigma$-algebra generated by $\mathcal P^{k}$. $(\mathcal{X},\mathcal{B}_\mathcal{X},\mu,f)$ is \textit{exponentially $\phi$-mixing} if there exist $C, c>0$  such that for every $H \in \mathcal{F}_{0,k}$ and every $A \in \mathcal{B}_{\mathcal{X}}$
	\begin{equation}\label{eq:exp-phi-mix}
		\modl{\mu(H \cap f^{-(t+k)}(A)) - \mu(H)\mu(A)} \leq C\e^{-ct}\mu(A).
	\end{equation}
\end{definition}

We remark  that $G$ is exponentially $\phi$-mixing, see for example \cite[Theorem~1]{AN05}.
Recall that, as observed above, we only need to check the conditions $\D_{q_n}^{*}$, $\D_{q_n}^{'*}$, where the set $A_{n,\ell}^{(q_n)}$ is replaced by the approximations $\Gamma_n^+$ and $\Gamma_n^-$. In what follows, we will only check such conditions for $\Gamma_n^+$, but the argument applies with minor adjustments to  $\Gamma_n^-$. 

\begin{proof}[Proof of Condition $\D_{q_n}^{*}$]
	We need to control:
	$$
	E_n:=\left\lvert \m\left(\Gamma_n^+ \cap \bigcap_{i=\ell}^{N} \mathscr{W}_{J_{n,i}}\left(\Gamma_n^+ \right)\right) - \m\left(\Gamma_n^+ \right)\m\left(\bigcap_{i=\ell}^{N} \mathscr{W}_{J_{n,i}}\left(\Gamma_n^+ \right)\right) \right\rvert 
	$$
	
	We start by noting that $E_n\leq 2\m(\Gamma_n^+)$. Also, observe that $\Gamma_n^{+}\in\mathcal{F}_{0,2j_{n,\tau'}}$.  Assuming that $t>2j_{n,\tau'}$ and  writing (\ref{eq:exp-phi-mix}) with $H=\Gamma_n^{+}$ and $f^{-(t-2j_{n,\tau'})+2j_{n,\tau'}}(A)=\bigcap_{i=l}^{m} \mathscr{W}_{J_{n,i}}(\Gamma_n^{+})$, we estimate $E_n\leq  C\e^{-c(t-2j_{n,\tau'})}$. Hence, we can set $\gamma(n,t)= \max\left\{2,C\e^{-c(t-2j_{n,\tau'})}\right\}$.
	
	Since $j_{n,\tau'}$ grows logarithmically in $n$, we can choose $t_n=o(n)$ and satisfying conditions \eqref{eq:kn-sequence} so that $\lim_{n\to\infty}n\gamma(n,t_n)=0$.	
\end{proof}

\begin{proof}[Proof of Condition $\D_{q_n}^{'*}$]
	We start by noting that, by construction, the points of $U_n(\tau', \tau'')^{(1)}$ are sent by $G$ into the holes that appeared in the $(j_{n,\tau'}-1)$-th and/or the $(j_{n,\tau'}-2)$-th step of the construction of $\Lambda$ and then they are sent successively by $G$ to lower and lower ranked holes in the construction until they reach the first holes of the construction and only then can they return to $U_n(\tau', \tau'')^{(1)}$. Using this fact and the exponential $\phi$ mixing property of $G$, we have:
	\begin{equation*}
		\begin{split}
			&n\m\left(\Gamma_n^{+} \cap \mathscr{W}^c_{[q_n+1,r_n)}\left(A_{n, (\bar\tau_i)_i, N} \right)\right)
			\leq n\sum_{j=j_{n,\tau'}-1}^{r_n-1} \m\left(\Gamma_n^{+} \cap f^{-j}\left(\Gamma_n^+\right)\right)\\
			&\leq n\sum_{j=j_{n,\tau'}-1}^{2j_{n,\tau'}} \m\left(\Gamma_n^+ \cap f^{-j}\left(\Gamma_n^+\right)\right)+  n\sum_{j=2j_{n,\tau'}}^{r_n-1} \m\left(\Gamma_n^{+} \cap f^{-j}\left(\Gamma_n^+\right)\right)\\
			&\leq n\sum_{j=j_{n,\tau'}-1}^{2j_{n,\tau'}} \m\left(\Gamma_n^+ \cap f^{-j}\left(\Gamma_n^+\right)\right)+ n\sum_{j=2j_{n,\tau'}}^{r_n-1} \m(\Gamma_n^{+})\m(\Gamma_n^{+}) + n \m(\Gamma_n^{+})\sum_{j=2j_{n,\tau'}}^{r_n-1}C\e^{-cj} \\
			&\leq n\sum_{j=j_{n,\tau'}-1}^{2j_{n,\tau'}} \m\left(\Upsilon^+(U_n(\tau', \tau'')^{(1)}, j) \cap f^{-j}\left(\Gamma_n^+\right)\right)+\tilde C\frac{n^2\left(\m(\Gamma_n^{+})\right)^2}{k_n}+\hat C n\m(\Gamma_n^{+})\sum_{j=2j_{n,\tau'}}^{\infty}C\e^{-cj},			\end{split}
	\end{equation*}
for some $\tilde C, \hat C>0$. The second term goes to $0$ by \eqref{eq:un}, Lemma~\ref{lem:cantor-cyl-approx} and the fact that $k_n\to\infty$. The third term goes to $0$ by \eqref{eq:un} and the fact that $j_{n,\tau'}\to\infty$.

Now, to handle the first term, we write $\Upsilon^+(U_n(\tau', \tau'')^{(1)}, j)=\bigcup_{\omega\in\mathcal P_j\colon \omega\subset\Upsilon^+(U_n(\tau', \tau'')^{(1)}, j)}\omega$, where the union is disjoint. Since $G$ is a full branched linear Markovian map, we have
$$
\m\left(\omega\cap f^{-j}\left(\Gamma_n^+\right)\right)\leq C \m(\omega)\m\left(\Gamma_n^+\right).
$$
It follows that 
\begin{equation*}
\begin{split}
 n\sum_{j=j_{n,\tau'}-1}^{2j_{n,\tau'}}& \m\left(\Upsilon^+(U_n(\tau', \tau'')^{(1)}, j)\cap f^{-j}\left(\Gamma_n^+\right)\right)\\
 &\leq  n\sum_{j=j_{n,\tau'}-1}^{2j_{n,\tau'}} \sum_{\omega\in\mathcal P_j\colon \omega\subset\Upsilon^+(U_n(\tau', \tau'')^{(1)}, j)}\m\left(\omega\cap f^{-j}\left(\Gamma_n^+\right)\right)\\
 &\leq  n\sum_{j=j_{n,\tau'}-1}^{2j_{n,\tau'}} \sum_{\omega\in\mathcal P_j\colon \omega\subset\Upsilon^+(U_n(\tau', \tau'')^{(1)}, j)}C \m(\omega)\m\left(\Gamma_n^+\right)\\
 &\leq  Cn\sum_{j=j_{n,\tau'}-1}^{2j_{n,\tau'}}  \m\left(\Upsilon^+(U_n(\tau', \tau'')^{(1)}, j)\right)\m\left(\Gamma_n^+\right)\leq \hat C\frac{j_{n,\tau'}}{n}\xrightarrow[n\to\infty]{}0
 \end{split}
\end{equation*}
	\end{proof}

\subsection{Anchored tail process}\label{sec:cantor-piling}
We will see that for our $G$ and $\varphi_{\alpha}$ the Cantor set $\Lambda$ behaves like a fixed point.

\begin{proposition}\label{prop:cantor-piling}
Let $G$ be as defined in Section \ref{sec:cantor-set} and let $\varphi_{\alpha}$ be as in Definition \ref{def:cantor-phi}. Then, the anchored tail process is (a.s.) the bi-infinite sequence $(Z_j)_{j \in \mathbb{Z}}$ with entry $U.(1-\theta)^{-j}$ at $j \in \mathbb{N}_0$ and $\infty$ otherwise, where $U$ is uniformly distributed on $[0,1]$ and $\theta$ is the extremal index (\textit{cf.} Proposition \ref{prop:cantor-EI}).
\end{proposition}
\begin{proof}
We check that the process $(Y_j)_{j \in \mathbb{Z}}$ is (a.s.) the bi-infinite sequence with:
\begin{enumerate}[(i)]
\item entry $U$ at $j=0$;
\item entries $U.(1-\theta)^{-j}$ for all positive indices $j$;
\item $\infty$ for all negative indices $j$ except, possibly, $U.(1-\theta)^{-j}$ at $j \geq -m$ for some $m \in \mathbb{N}$;
\end{enumerate}
where $U$ is uniformly distributed on $[0,1]$.

Observe that $U.(1-\theta)^{-j}<1$ for all $j<0$. Therefore, if $(Y_j)_{j \in \mathbb{Z}}$ is as described by (i)-(iii) then $(Z_j)_{j \in \mathbb{Z}}$ is as in the statement of the theorem.

We must check that conditions (1)-(4) in Section~\ref{ssec:seqtail} are satisfied for $(Y_j)_{j \in \mathbb{Z}}$ as given by (i)-(iii).

Conditions (2) and (3) are straightforward.

As for condition (4), $Y_j \geq 1$ for all $j \leq -1$ implies that the hit to $\Lambda_{j_{n,\tau}} \cup H^{(j_{n,\tau})}$ at time $r_n$ is the first hit to that same neighbourhood of $\Lambda$ (\textit{i.e.} there is no hit to $\Lambda_{j_{n,\tau}} \cup H^{(j_{n,\tau})}$ before time $r_n$ given that there is a hit to $\Lambda_{j_{n,\tau}} \cup H^{(j_{n,\tau})}$ at time $r_n$ - recall, from (\ref{eq:cantor-U_n}), that $\{X_{r_n}>u_n(\tau)\}=G^{-r_n}(\Lambda_{j_{n,\tau}} \cup H^{(j_{n,\tau})})$). Since
\begin{equation*}
\dfrac{\m(G^{-r_n}(\Lambda_{j_{n,\tau}} \cup H^{(j_{n,\tau})}) \cap G^{-(r_n-1)}([0,1] \setminus (\Lambda_{j_{n,\tau}} \cup H^{(j_{n,\tau})})))}{\m(G^{-r_n}(\Lambda_{j_{n,\tau}} \cup H^{(j_{n,\tau})}))}>0,
\end{equation*}
 condition (4) is satisfied.

So, we check that condition (1) holds. Recall that $u_n^{-1}(X_j)=\left(\dfrac{X_j}{\mathfrak{a}_n}\right)^{-\alpha}$. Therefore,
\begin{equation*}
\left\lbrace\dfrac{1}{\tau}u_n^{-1}(X_{r_n+j}) \given X_{r_n}>u_n(\tau)\right\rbrace=\left\lbrace\dfrac{1}{\tau}\left(\dfrac{X_{r_n+j}}{\mathfrak{a}_n}\right)^{-\alpha} \given X_{r_n}>u_n(\tau)\right\rbrace.
\end{equation*}

Now, $\{x \in [0,1]: X_{r_n}(x)>u_n(\tau)\}=G^{-r_n}(\Lambda_{j_{n,\tau}} \cup H^{(j_{n,\tau})})$.

First note that for $t\in [0, 1]$, $\p(Y_0\le t)  =\p(F\circ\psi\le t)=t$ by Lemma~\ref{lem:proper-psi}(c)(iii), so we write this uniform distribution by $U$.  For $j \geq 1$, if $G^{r_n}(x) \in \Lambda_{j_{n,\tau}} \cup H^{(j_{n,\tau})}$ then $G^{r_n+j}(x) \in \Lambda_{j_{n,\tau}-j} \cup H^{(j_{n,\tau}-j)}$ provided\\ $j_{n,\tau}-j \geq 0$, in which case $\psi(G^{r_n+j}(x))=2^j\psi(G^{r_n}(x))$ giving
\begin{equation*}
\begin{split}
\dfrac{F(\psi(G^{r_n+j}(x)))}{F(\psi(G^{r_n}(x)))}=\dfrac{\mathbb{P}(\psi \leq \psi(G^{r_n+j}(x)))}{\mathbb{P}(\psi \leq \psi(G^{r_n}(x)))}=\dfrac{\m(\Lambda_{j_{n,\tau}-j} \cup H^{(j_{n,\tau}-j)})}{\m(\Lambda_{j_{n,\tau}} \cup H^{(j_{n,\tau})})}&=\left(\displaystyle\sum_{i=1}^{k}\dfrac{1}{m_i}\right)^{-j}\\
&=(1-\theta)^{-j}
\end{split}
\end{equation*}

We have
\begin{equation*}
\lim_{n \to \infty}\mathbb{P}\left(\left\lbrace\dfrac{u_n^{-1}(X_{r_n+j})}{u_n^{-1}(X_{r_n})}=(1-\theta)^{-j} \given X_{r_n}>u_n(\tau)\right\rbrace\right)=1
\end{equation*}
(since $j_{n,\tau}-j \geq 0$ when $n \to \infty$).

We may write
\begin{equation*}
\begin{split}
& \hspace{-1cm} \lim_{n \to \infty}\mathbb{P}\left\lbrace\dfrac{1}{\tau}u_n^{-1}(X_{r_n+j})=s \given X_{r_n}>u_n(\tau)\right\rbrace\\
&=\lim_{n \to \infty}\mathbb{P}\left\lbrace\dfrac{u_n^{-1}(X_{r_n+j})}{u_n^{-1}(X_{r_n})}\dfrac{u_n^{-1}(X_{r_n})}{\tau}=s \given X_{r_n}>u_n(\tau)\right\rbrace
\end{split}
\end{equation*}
which, by (2) in in Section~\ref{ssec:seqtail}  holds, gives $\mathcal{L}(Y_j)=\mathcal{L}(\Theta_jY_0)=(1-\theta)^{-j}.U$ (a.s.).

Let $j<0$. If $G^{r_n}(x) \in \Lambda_{j_{n,\tau}} \cup H^{(j_{n,\tau})}$ then $G^{r_n+j}(x) \in \Lambda_{j_{n,\tau}-j} \cup H^{(j_{n,\tau}-j)}$ for an at most finite number of indices $j \geq -m$ where $m \in \mathbb{N}$. In words, a hit, at time $r_n$, to the approximation (of $\Lambda$) $\Lambda_{j_{n,\tau}} \cup H^{(j_{n,\tau})}$ can only be preceded by a finite number of hits to the finer approximations $\Lambda_{j_{n,\tau}-j} \cup H^{(j_{n,\tau}-j)}$, as, otherwise, the point would be in $\Lambda$. This leads to $Y_j=U.(1-\theta)^{-j}$ for an at most finite number of indices $j \geq -m$ where $m \in \mathbb{N}$.

If $G^{r_n}(x) \in \Lambda_{j_{n,\tau}} \cup H^{(j_{n,\tau})}$ is such that $G^{r_n+k}(x) \in [0,1] \setminus \Lambda_1$ for some $k \in \mathcal{K} \subseteq \{1,\dots,-j\}$, we have
\begin{equation*}
\dfrac{\psi(G^{r_n+j}(x))}{\psi(G^{r_n}(x))}=\dfrac{\sum_{k \in \mathcal{K}} 2^{-k} + 2^j\psi(G^{r_n}(x))}{\psi(G^{r_n}(x))}
\end{equation*}
so that
\begin{equation*}
\dfrac{F(\psi(G^{r_n+j}(x)))}{F(\psi(G^{r_n}(x)))}=\dfrac{\m([0,1] \setminus \Lambda_1)}{\m(\Lambda_{j_{n,\tau}} \cup H^{(j_{n,\tau})})}.
\end{equation*}

This leads to
\begin{equation*}
\lim_{n \to \infty}\mathbb{P}\left\lbrace\dfrac{u_n^{-1}(X_{r_n+j})}{u_n^{-1}(X_{r_n})}=\infty \given X_{r_n}>u_n(\tau)\right\rbrace=1.
\end{equation*}
So, $Y_j=\infty$.
\end{proof}

\subsection{The $1< \alpha<2$ case}  
To prove our result in the case $1< \alpha<2$, we are required to satisfy the small jumps condition \eqref{eq:small-jumps} and
we will also need to check that the sequence $(Q_j)_{j\in\Z}$, obtained from the spectral decomposition of the transformed anchored tail process, satisfies the assumption \eqref{eq:Q-condition->1}.
Since we have geometric decay of $\mathcal{ Q}_j$ (from geometric growth of $(\mathbf{\tilde{Q}}_{i, j})_j$), this latter condition holds in our case.

For the small jumps condition, we will now follow the proof strategy of \cite{CheNicTor24}.  Note that they did not cover the case $\alpha=1$, so we will also omit this. 
As is well known (see \cite{D83}), to show \eqref{eq:small-jumps} it is sufficient to show that for $Y_j= X_j\1_{X_j\leq \varepsilon \mathfrak{a}_n}- \E\left(X_j\1_{X_j\leq \varepsilon \mathfrak{a}_n}\right)$, that 
$$\lim_{\eps\to 0}\limsup_{n\to \infty} \frac n{\mathfrak{a}_n^2}\sum_{j=1}^n\max\left\{0, \E(Y_1Y_j)\right\}=0.$$

As in \cite{CheNicTor24}, we define $\phi_{n,\eps} = \phi_\alpha\cdot \1_{ \phi_\alpha\leq \varepsilon \mathfrak{a}_n}$ and it is sufficient to show that

$$\lim_{\eps\to 0}\limsup_{n\to \infty} \frac n{\mathfrak{a}_n^2}\sum_{j=1}^n\int |\phi_{n, \eps}|\cdot |\phi_{n, \eps}\circ G^j|~d\m=0.$$

In that paper they use the decay of correlations against $BV$, which we also have here.  By  Lemma~\ref{lem:proper-psi}(c)(ii) that the set $\{\phi_\alpha<t\} = \Lambda_k$ for $k\approx \frac{-\alpha\log t}{\log\lambda}$, so the complexity here is polynomial.  That is, there exists $\beta>1$ such that 
$$\|\1_{\{\phi_\alpha<t\}}\|_{BV} \le Kt^\beta.$$

This is useful for us since if we have exponential decay of correlations at rate $\theta$, as in \cite{CheNicTor24} we need only show
$$\lim_{\eps\to 0}\limsup_{n\to \infty} \frac n{\mathfrak{a}_n^2}\sum_{j=1}^{b\log n}\int |\phi_{n, \eps}|\cdot |\phi_{n, \eps}\circ G^j|~d\m=0$$
since we can choose $b$ large enough so that 
$$\frac{n}{\mathfrak{a}_n^2} \|\phi_{n, \eps}\|_{BV}\|\phi_{n, \eps}\|_{L^1}e^{-cj} \lesssim \frac{n}{\mathfrak{a}_n^2} (\mathfrak{a}_n\eps)^{\beta+1}e^{-\theta j} \approx n^{1+\frac1\alpha (\beta-1)} e^{-\theta j} $$
is bounded for $j= b\log n$.

For the rest of the argument, we can follow the steps of the proof of \cite[Theorem 8.1]{CheNicTor24}.   For the estimates (II) and (III) there, the argument is identical.  For (I) we observe that in the setting of this paper we are always in Case 1, and the derivative is constant 3, so we get an exponentially decaying constant in front of $\int\phi_{n, \eps}^2~d\m$, so the proof then follows.

\subsection{Conclusion of Theorem~\ref{thm:REPP-conv}}

We have shown that our observable has regularly varying tails, a well-defined anchored process, that the system satisfies the relevant dependence/mixing conditions and in the case $1<\alpha<2$, we have the required small jumps condition.  Hence Theorem~\ref{thm:REPP-conv} (and hence, as in \cite{FFT25}, also Theorem~\ref{thm:enriched-conv}) is proved.

\end{document}